\documentclass[12pt,reqno]{amsart}
\usepackage{amsmath,amssymb,amsthm,mathtools,calc,verbatim,enumitem,tikz,url,hyperref,mathrsfs,cite,fullpage} 

\newtheorem{theorem}{Theorem}[section]
\newtheorem{lemma}[theorem]{Lemma}
\newtheorem{corollary}[theorem]{Corollary}

\newtheorem{prop}[theorem]{Proposition}

\theoremstyle{definition}

\newtheorem*{qu*}{Question}
\theoremstyle{remark}

\newcommand\N{\mathbb{N}}
\newcommand\R{\mathbb{R}}
\newcommand\Z{\mathbb{Z}}
\newcommand\E{\operatorname{\mathbb{E}}}
\newcommand\cA{\mathcal{A}}
\newcommand\cB{\mathcal{B}}
\newcommand\cF{\mathcal{F}}
\newcommand\cH{\mathcal{H}}
\newcommand\cN{\mathcal{N}}
\newcommand\cP{\mathcal{P}}

\renewcommand\Pr{\operatorname{\mathbb{P}}}
\newcommand\id{\hbox{$1\mkern-6.5mu1$}}

\newcommand\eps{\varepsilon}

\renewcommand\le{\leqslant}
\renewcommand\ge{\geqslant}
\renewcommand\to{\rightarrow}
\newcommand\ds{\displaystyle}

\begin{document}

\title{The Erd\H{o}s--Selfridge problem with square-free moduli}
\author{Paul Balister \and B\'ela Bollob\'as \and Robert Morris \and \\
Julian Sahasrabudhe \and Marius Tiba}

\address{Department of Mathematical Sciences,
University of Memphis, Memphis, TN 38152, USA}\email{pbalistr@memphis.edu}

\address{Department of Pure Mathematics and Mathematical Statistics, Wilberforce Road,
Cambridge, CB3 0WA, UK, and Department of Mathematical Sciences,
University of Memphis, Memphis, TN 38152, USA}\email{bb12@cam.ac.uk}

\address{IMPA, Estrada Dona Castorina 110, Jardim Bot\^anico,
Rio de Janeiro, 22460-320, Brazil}\email{rob@impa.br}

\address{Peterhouse, Trumpington Street, University of Cambridge, CB2 1RD, UK
\and IMPA, Estrada Dona Castorina 110, Jardim Bot\^anico,
Rio de Janeiro, 22460-320, Brazil }\email{julians@impa.br}

\address{Department of Pure Mathematics and Mathematical Statistics, Wilberforce Road,
Cambridge, CB3 0WA, UK}\email{mt576@dpmms.cam.ac.uk}

\thanks{The first two authors were partially supported by NSF grant DMS 1600742,
the third author was partially supported by CNPq (Proc.~303275/2013-8) and
FAPERJ (Proc.~201.598/2014), and the fifth author was supported by a Trinity Hall Research Studentship.}


\begin{abstract}
A \emph{covering system} is a finite collection of arithmetic progressions whose union is the set of integers. The study of covering systems with distinct moduli was initiated by Erd\H{o}s in 1950, and over the following decades numerous problems were posed regarding their properties. One particularly notorious question, due to Erd\H{o}s, asks whether there exist covering systems whose moduli are distinct and all odd. We show that if in addition one assumes the moduli are square-free, then there must be an even modulus.
\end{abstract}

\maketitle

\section{Introduction}

Almost 70 years ago, Erd\H{o}s~\cite{E50} initiated the study of \emph{covering systems}, i.e., finite collections of arithmetic progressions\footnote{We exclude the trivial arithmetic progression $\Z$.} that cover the integers, with distinct moduli. Many well-known questions and conjectures have been posed about such systems (some of which appeared frequently in Erd\H{o}s's collections of open problems), and in recent years there has been significant progress on several of these. A first crucial step was taken in 2007, by Filaseta, Ford, Konyagin, Pomerance and Yu~\cite{FFKPY}, who proved that the sum of the reciprocals of the moduli grows quickly with the minimum modulus, and also confirmed a conjecture of Erd\H{o}s and Graham~\cite{EG} on the density of the uncovered set. A further important breakthrough was made in 2015, by Hough~\cite{H}, who resolved the so-called `minimum modulus problem' of Erd\H{o}s~\cite{E50} by showing that the minimum modulus is bounded. More recently, the current authors~\cite{DUS} developed a general method (based on that of~\cite{H}) for attacking problems of this type, and used it to study the density of the uncovered set, and to prove a conjecture of Schinzel~\cite{Sch} by showing that there must exist two moduli, one of which divides the other.

In this paper we will further develop the method of~\cite{DUS} in order to make progress on another old and well-known question: does there exist a covering system whose moduli are distinct and all odd? This question appears to have first been asked by Erd\H{o}s~\cite{E65} in 1965, and a few years later he conjectured (see~\cite{E73a}) that there does exist such a system. In 1977 he went further, conjecturing~\cite{E77} that there exist covering systems with square-free moduli, all of whose prime factors are arbitrarily large. On the other hand (as recounted, for example, in~\cite{FFK}), Selfridge believed that there do \emph{not} exist such systems, and (perhaps as a result) the question has become known as the Erd\H{o}s--Selfridge problem. Apart from its intrinsic appeal, the problem is motivated by a theorem of Schinzel~\cite{Sch}, who discovered a connection between the non-existence of such covering systems and the irreducibility of certain polynomials. More precisely, he showed that if no covering system with distinct, odd moduli exists, then for every polynomial $f(x) \in \Z[X]$ with $f \not\equiv 1$, $f(0) \ne 0$ and $f(1) \ne -1$, there exists an infinite arithmetic progression of values of $n \in \Z$ such that $x^n + f(x)$ is irreducible over the rationals.

The first progress on the Erd\H{o}s--Selfridge problem was made by Simpson and Zeilberger~\cite{SZ}, who proved that the moduli of a covering system with distinct, odd, square-free numbers use at least 18 primes (this was later improved to 22 primes by Guo and Sun~\cite{GS}). A major step forward was taken by Hough and Nielsen~\cite{HN}, who used a refined (and carefully optimised) version of the method of Hough~\cite{H} to prove that every covering system with distinct moduli contains a modulus that is divisible by either~2 or~3. The general method of~\cite{DUS} (which, as noted above, is also based on that of~\cite{H}) provides a short proof of the following slight strengthening of this result (see~\cite[Theorem~1.4]{DUS}): every covering system with distinct moduli contains either $(a)$ an even modulus, $(b)$ a modulus divisible by $3^2$, or $(c)$ (possibly equal) moduli $d_1$ and $d_2$ with $3 \mid d_1$ and $5 \mid d_2$. Here we will further develop the method of~\cite{DUS}, and use it to solve the Erd\H{o}s--Selfridge problem in the square-free case.

\begin{theorem}\label{thm:SFES}
In any finite collection of arithmetic progressions with distinct square-free moduli that covers the integers, at least one of the moduli is even.
\end{theorem}

We shall prove Theorem~\ref{thm:SFES} in a (slightly more general) geometric setting; a second aim of this paper will be to investigate covering systems in this setting. Let $S_1,\ldots,S_n$ be finite sets with at least two elements, and set
$$Q = S_1 \times \cdots \times S_n.$$
If $A = A_1 \times\cdots \times A_n \subseteq Q$ with each $A_k$ either equal to $S_k$ or a singleton element of $S_k$, then we say that $A$ is a {\em hyperplane}. We will write $A = [x_1,\ldots,x_n]$, where $x_k \in S_k \cup \{ * \}$ for each $k \in [n]$, and $*$ indicates that $A_k = S_k$. Let us write $F(A) = \{ k : x_k \in S_k \}$ for the set of \emph{fixed coordinates} of $A$, and say that two hyperplanes $A$ and $A'$ are \emph{parallel\/} if $F(A) = F(A')$.

Restating Theorem~\ref{thm:SFES} in this geometric setting gives the following theorem.


\begin{theorem}\label{thm:GSFES}
For each\/ $k \in [n]$, let\/ $p_k$ be the $k$th prime, and set\/ $S_k = [p_{k+1}]$. Any collection of hyperplanes that covers\/ $Q := S_1 \times \dots\times S_n$ contains two parallel hyperplanes.
\end{theorem}

To see the equivalence between Theorems~\ref{thm:SFES} and~\ref{thm:GSFES}, note that, by the Chinese Remainder Theorem, there is a natural equivalence\footnote{To be precise, the progression $a + d\Z$ with $d = \prod_{i \in I} p_i$ corresponds to the hyperplane $A = [x_1,\ldots,x_n]$ where $x_i = a \bmod p_i$ if $i \in I$, and $x_i = *$ otherwise. Note that excluding $\Z$ is equivalent to forbidding $F(A) = \emptyset$.}  between finite collections $\cA$ of arithmetic progressions with square-free, odd, $p_{n+1}$-smooth moduli that cover the integers, and finite collections $\cH$ of hyperplanes that cover the box $Q = [p_2] \times \cdots \times [p_{n+1}]$. Moreover, if the moduli of $\cA$ are distinct then the hyperplanes in $\cH$ are non-parallel .

In order to motivate our second main theorem, let us next state, in this geometric setting, a special case (for square-free moduli) of the breakthrough result of Hough~\cite{H} which resolved the Erd\H{o}s minimum modulus problem.

\begin{theorem}[Hough, 2015]\label{thm:MM}
Let\/ $p_1,\ldots,p_n$ be the first $n$ primes. There exists a constant\/ $C$ such that if\/ $\cA$ is a collection of hyperplanes that cover $Q := [p_1] \times \dots\times [p_n]$, then either two of the hyperplanes are parallel, or there exists a hyperplane\/ $A \in \cA$ with $F(A) \subseteq [C]$.
\end{theorem}

To deduce Theorem~\ref{thm:MM} from Hough's theorem, simply note that if $d$ is square-free and $p_C$-smooth, then $d \le M := \prod_{i = 1}^C p_i$. Using our method, we are able to prove the following strengthening of Theorem~\ref{thm:MM}.

\begin{theorem}\label{thm:GMM}
For every sequence of integers\/ $(q_k)_{k \ge 1}$ such that $q_k \ge 2$ for each $k \in \N$ and
$$\liminf_{k \to \infty} \frac{q_k}{k} > 3,$$
there exists a constant $C$ such that the following holds. Let\/ $\cA$ be a collection of hyperplanes that cover $Q := [q_1] \times \dots\times [q_n]$ for some $n \in \N$. Then either two of the hyperplanes are parallel, or there exists a hyperplane\/ $A \in \cA$ with $F(A) \subseteq [C]$.
\end{theorem}

Note that in Theorem~\ref{thm:MM} the sequence $(p_k)_{k \ge 1}$ grows asymptotically as $k \log k$, whereas in Theorem~\ref{thm:GMM} we allow the sequence $(q_k)_{k\ge 1}$ to grow only linearly. We will show (see Section~\ref{sec:proofGMM}) that Theorem~\ref{thm:GMM} is close to best possible, since there exists an example with $\liminf q_k/k = 1$ for which the conclusion of the theorem fails.

The rest of the paper is organized as follows. In Section~\ref{sec:sieve} we outline the sieve that 
we will use in the proofs, and in Section~\ref{sec:general} we state and prove our main technical
results, Theorems~\ref{thm:general} and~\ref{prop:moments}. 
In Section~\ref{sec:proofGMM} we deduce Theorem~\ref{thm:GMM}. Finally, we dedicate Section~\ref{sec:proofSFES} to the proof of our main result, Theorem~\ref{thm:SFES}.

\section{Definition of the Sieve}\label{sec:sieve}

In this section we will outline the proofs of Theorems~\ref{thm:GSFES} and~\ref{thm:GMM}. In particular,
we will generalize the method developed in \cite{DUS} to the geometric setting, and while doing so we will introduce several new ideas that will prove to be crucial in the proofs. For the convenience of the reader and for completeness, we will include full proofs of all intermediate results, even though several of them are direct adaptations of the corresponding results in~\cite{DUS}.

As in the Introduction, let $S_1,\ldots,S_k$ be finite sets with at least two elements, and set
$$Q := S_1 \times \cdots \times S_n.$$
and let $\cA$ be a collection of hyperplanes, no two of which are parallel. Set
$$\cF = \cF(\cA) := \big\{ F(A) : A \in \cA \big\} \subseteq \cP([n]) \setminus \{\emptyset\},$$
and (recalling that $F(A) \ne F(A')$ for distinct $A,A' \in \cA$) let us index the hyperplanes in $\cA$ by the corresponding set of fixed coordinate indices, so $\cA = \{A_F : F \in \cF \}$. Our goal is to estimate the density (under some probability measure) of the uncovered set
$$R := Q \setminus \bigcup_{F\in \cF} A_F.$$
Rather than considering the entire collection of hyperplanes $\cA$ all at once, we expose the
hyperplanes dimension by dimension and track how the density of the uncovered set evolves.
To be more precise, define, for each $1 \le k \le n$,
$$\cF_k := \big\{ F \in \cF: F \subseteq [k] \big\} \qquad \text{and} \qquad \cA_k := \big\{ A_F : F \in \cF_k \big\}$$
for the family of sets of fixed coordinate indices and the corresponding hyperplanes that are contained
in the initial segment~$[k]$. Let
$$R_k := Q \setminus \bigcup_{F \in \cF_k} A_F = Q \setminus \bigcup_{A_F \in \cA_k} A_F,$$
be the set of elements not contained in any of the hyperplanes of~$\cA_k$, so in particular $R_n = R$. We also write $\cN_k := \cF_k \setminus \cF_{k-1}$ for the family of ``new'' sets of fixed coordinate indices at the $k$th stage, i.e., those sets that contain $k$ and are contained in $[k]$, and define
\begin{equation}\label{def:B_i}
B_k := \bigcup_{F \in \cN_k} A_F
\end{equation}
to be the union of the hyperplanes exposed at step~$k$, so that $R_k = R_{k-1} \setminus B_k$.

It will often be convenient to consider $R_k$, $B_k$ and $A_F$ with $F\in \cF_k$ as subsets of
$$Q_k := S_1 \times \cdots \times S_k$$
by identifying $X \subseteq Q_k$ with $X \times S_{k+1} \times \dots \times S_n$. We call a set of this form $Q_k$-{\em measurable}.

\subsection{The probability measures $\Pr_k$}

The construction of the probability measures is similar to that in~\cite{DUS}, and no significant new ideas are needed. The main difference from~\cite{DUS} is that instead of starting with the uniform measure as our $\Pr_0$, we allow for possible optimization of the measure on the first few coordinates. In general we will start with some measure $\Pr_a$, to be determined, which will be supported on $R_a\subseteq Q_a= S_1 \times \cdots \times S_a$.

Our aim is to construct, for each $a < k \le n$, a measure $\Pr_k$ on $Q_k$ in such a way that $\Pr_k(B_k)$ is small, but without changing the measure of $B_i$ for any $i < k$. Fix a sequence of constants $\delta_{a+1},\dots,\delta_n \in [0,1/2]$, and assume that we have already defined a probability measure $\Pr_{k-1}$ on $Q_{k-1}$. Recall that $Q_k = Q_{k-1} \times S_k$, and hence the elements of $Q_k$ can be written as pairs $(x,y)$, where $x \in Q_{k-1}$ and $y \in S_k$. We may view $R_{k-1}$ as a collection of fibres of the form $F_x = \{ (x,y) : y \in S_k \} \subseteq Q_k$, where $\Pr_{k-1}$ is extended uniformly to a measure on $Q_k$ (so is uniform on each fibre), and view $R_k$ as being obtained from $R_{k-1}$ by removing $B_k$, i.e., by removing the points that are contained in the new hyperplanes of $\cA_k \setminus \cA_{k-1}$.

Now, for each $x \in Q_{k-1}$, define
\begin{equation}\label{def:alpha}
 \alpha_k(x) = \frac{\Pr_{k-1}\big( F_x \cap B_k \big)}{\Pr_{k-1}(x)}
 = \frac{\big| \big\{ y \in S_k : (x,y)\in B_k \big\} \big|}{|S_k|},
\end{equation}
that is, the proportion of the fibre $F_x$ that is removed at stage~$k$. The probability measure~$\Pr_k$ on $Q_k$ is defined as follows:
\begin{equation}\label{def:P_k}
 \Pr_k(x,y) :=
 \begin{cases}
   \max\bigg\{ 0, \, \ds\frac{\alpha_k(x)-\delta_k}{\alpha_k(x)(1-\delta_k)} \bigg\} \cdot \Pr_{k-1}(x,y),
   &\text{if }(x,y)\in B_k;\\[2ex]
   \min\bigg\{ \ds\frac{1}{1-\alpha_k(x)}, \, \frac{1}{1-\delta_k} \bigg\} \cdot \Pr_{k-1}(x,y),
   &\text{if }(x,y)\notin B_k.
 \end{cases}
\end{equation}
To motivate the definition above, note that if $\alpha_k(x) \le \delta_k$, then $\Pr_k(x,y) = 0$ for every
element of $Q_k$ that is covered in step~$k$, and that the measure is increased proportionally
elsewhere to compensate. On the other hand, for those $x \in Q_{k-1}$ for which $\alpha_k(x) > \delta_k$, we `cap' the distortion by increasing the measure at each point not covered in step $k$ by a factor of $1 / (1-\delta_k)$, and decreasing the measure on removed points by a corresponding factor.

%



The measure $\Pr_k$ satisfies the following simple properties, cf.~\cite[Lemmas~2.1 and~2.2]{DUS}.

\begin{lemma}\label{obs:probs}
 For any\/ $k > a$ and any\/ $Q_{k-1}$-measurable set\/~$S$ we have
 \begin{equation}\label{obs:Qmeas}
  \Pr_k(S) = \Pr_{k-1}(S).
 \end{equation}
 For any set\/ $S \subseteq Q$, we have
 \begin{equation}\label{e:dom}
 \Pr_k(S) \le \frac{1}{1-\delta_k} \cdot \Pr_{k-1}(S).
 \end{equation}
 Moreover, if\/ $S \subseteq B_k$ then
\begin{equation}\label{e:baddom}
  \Pr_k(S) \le \Pr_{k-1}(S).
   \end{equation}
\end{lemma}


In particular, it follows from Lemma~\ref{obs:probs} that if 
\begin{equation}\label{eq:sufficient}
\sum_{k = a+1}^n \Pr_k(B_k) < 1
\end{equation}
then $\cA$ does not cover $Q$, since $B_k$ is a $Q_k$-measurable set, so by~\eqref{obs:Qmeas} we have $\Pr_n(B_k) = \Pr_k(B_k)$. For each $a \le k \le n$, define
$$\mu_k := 1 - \sum_{i = a+1}^k \Pr_i(B_i),$$
and observe that $\mu_k \le \Pr_k(R_k)$. 


\section{A general theorem}\label{sec:general}

In this section we will prove two technical results, Theorems~\ref{thm:general} and~\ref{prop:moments}, which together imply Theorems~\ref{thm:GSFES} and~\ref{thm:GMM}. We remark that Theorem~\ref{thm:general} essentially follows from~\cite[Theorem~3.1]{DUS}, but Theorem~\ref{prop:moments} introduces a new bound that is motivated geometrically, and that will prove to be crucial in the proof of Theorem~\ref{thm:GSFES}.

Given a collection $\cA$ of hyperplanes in $Q = S_1 \times \cdots \times S_n$, a probability distribution $\Pr_a$ on $Q_a$, and constants $\delta_{a+1},\dots,\delta_n \in [0,1/2]$, let the probability distributions $\Pr_k$ and functions $\alpha_k \colon Q_{k-1} \to [0,1]$ be defined as in~\eqref{def:alpha} and~\eqref{def:P_k}, and set 
\[
 M_k^{(1)} := \E_{k-1}\big[\alpha_k(x)\big] \qquad \text{and} \qquad
 M_k^{(2)} := \E_{k-1}\big[\alpha_k(x)^2\big].
\]
In order to show that\/ $\cA$ does not cover\/~$Q$, it is sufficient, by~\eqref{eq:sufficient}, to show that $\mu_n > 0$. To do so, we will bound $\Pr_k(B_k)$ in terms of the moments $M_k^{(1)}$ and $M_k^{(2)}$. As noted above, the following theorem was (essentially) proved in~\cite{DUS}.

\begin{theorem}\label{thm:general}
Let\/ $\cA$ be a collection of hyperplanes in $Q = S_1 \times \dots\times S_n$, no two of which are parallel. If
 \begin{equation}\label{eq:eta}
 \sum_{k = a+1}^n \min\bigg\{ M_k^{(1)}, \frac{M_k^{(2)}}{4\delta_k(1-\delta_k)} \bigg\}  <  1,
 \end{equation}
 then\/ $\cA$ does not cover\/~$Q$.
\end{theorem}

In order to show that~\eqref{eq:eta} holds in our applications, we need to bound the moments of~$\alpha_k(x)$. To state our bounds on $M_k^{(1)}$ and $M_k^{(2)}$, we will need some additional notation.  Define a function $c \colon \cP([1,a])\to [0,1]$ by setting
\begin{equation}\label{def:c}
c(I)= \max \big\{\Pr_a(H) : H \text{ is a hyperplane in $Q_a$ with } F(H) = I \big\}
\end{equation}
for each $I \subseteq [a]$, and define a function $\nu\colon \cP([a+1,n]) \to \R_{>0}$, by setting
\begin{equation}\label{def:nu}
 \nu(J) = \prod_{j \in J} \frac{1}{(1-\delta_j)|S_j|}
\end{equation}
for each $J \subseteq [a+1,n]$. Note that $c(\emptyset) = \nu(\emptyset) = 1$. For each $k \ge a$ and $x \in \R$, set
\begin{equation}\label{def:ck}
c_k(x) = \sum_{I \subseteq [a]} \sum_{J \subseteq [a+1,k]} c(I)\nu(J) x^{|I|+|J|} = \, \sum_{I \subseteq [a]} c(I) x^{|I|} \prod_{j = a + 1}^k \bigg(1 + \frac{x}{(1-\delta_j)|S_j|} \bigg).
\end{equation}
The following technical theorem provides general bounds on $M_k^{(1)}$ and $M_k^{(2)}$.

\begin{theorem}\label{prop:moments}
Let\/ $\cA$ be a collection of hyperplanes in $Q = S_1 \times \dots\times S_n$, no two of which are parallel. Then, for each\/ $a < k \le n$,
\begin{equation}\label{eq:M:bounds}
M_k^{(1)} \le \frac{c_{k-1}(1)}{|S_k|} \qquad \text{and} \qquad M_k^{(2)} \le \frac{c_{k-1}(3)}{|S_k|^2}.
\end{equation}
Moreover, if in addition $|S_k| \ge 3$ for each $k \in [n]$, and none of the hyperplanes in\/ $\cN_k$ has co-dimension\/~$1$, then
\begin{equation}\label{eq:hyperplane:removal}
M_k^{(2)} \le \frac{1}{|S_k|^2} \big(c_{k-1}(3)-2 c_{k-1}(1)+1\big).
\end{equation}
\end{theorem}

Before embarking on the (straightforward) proofs of Theorems~\ref{thm:general} and~\ref{prop:moments}, let us briefly discuss the bound~\eqref{eq:hyperplane:removal}, which will play an important role in the proof of Theorem~\ref{thm:SFES}. In order to apply it, we first need to remove from $\cA$ each of the codimension\/~$1$ hyperplanes, each of which is of the form $S_1 \times \dots \times S_{i-1} \times \{s\}\times S_{i+1} \times \dots \times S_n$ for some $i \in [n]$ and $s \in S_i$. Note that in doing so we remove the point $s$ from the possible values of the $i$th coordinate, effectively replacing $S_i$ by $S'_i := S_i \setminus \{s\}$ (which has at least two elements). After removing these hyperplanes, the remaining elements of $\cA$ will all have at least two fixed coordinates, and can be assumed to be hyperplanes in $Q' = S_1' \times \cdots \times S_n'$, where $S_i' = S_i$ if $\{i\} \not\in \cF(\cA)$. Removing the codimension~$1$ hyperplanes in this way makes a significant difference to our estimate on $M_k^{(2)}$, at the expense of (possibly) reducing each $|S_i|$ by~1. In practice, this turns out to often give better bounds on the removed measure.

\begin{proof}[Proof of Theorem~\ref{thm:general}]
Observe first that
\[
 \Pr_n(B_k) = \Pr_k(B_k) \le \Pr_{k-1}(B_k) = \E_{k-1}\big[\alpha_k(x)\big],
\]
where the first two steps follow by Lemma~\ref{obs:probs} (since $B_k$ is $Q_k$-measuarable), and the third follows by the definition~\eqref{def:alpha} of $\alpha_k(x)$. Moreover, by~\eqref{def:alpha} and~\eqref{def:P_k} (the definitions of $\alpha_k$ and~$\Pr_k$), we have
\begin{align}
\Pr_k(B_k) & = \sum_{x \in Q_{k-1}} \max\bigg\{ 0, \, \frac{\alpha_k(x) - \delta_k}{\alpha_k(x)( 1 - \delta_k )} \bigg\} \cdot \Pr_{k-1}\big(F_x \cap B_k \big) \nonumber\\
& = \frac{1}{1-\delta_k} \sum_{x \in Q_{k-1}} \max\big\{ 0, \, \alpha_k(x)-\delta_k \big\} \cdot \Pr_{k-1}(x) \nonumber\\
& \le \frac{1}{1-\delta_k}\sum_{x \in Q_{k-1}} \frac{\alpha_k(x)^2}{4\delta_k} \cdot \Pr_{k-1}(x)
= \frac{\E_{k-1}\big[ \alpha_k(x)^2 \big]}{4\delta_k(1-\delta_k)}, \label{eq:bound:on:Bk}
\end{align}
where we used the elementary inequality $\max\{a-d,0\} \le a^2 / 4d$, which is easily seen to hold for all $a,d > 0$ by rearranging the inequality $(a-2d)^2 \ge 0$.

It follows that the uncovered set $R$ satisfies
\[
\Pr_n(R) = 1 - \sum_{k = a+1}^n \Pr_n(B_k) \ge 1 - \sum_{k = a+1}^n \min\bigg\{ M_k^{(1)}, \frac{M_k^{(2)}}{4\delta_k(1-\delta_k)} \bigg\} > 0,
\]
by~\eqref{eq:eta}, and hence $\cA$ does not cover $Q$, as required.
\end{proof}

In the proof of Theorem~\ref{prop:moments} we will use the following notation. Given a hyperplane $A = [x_1,\ldots,x_n]$ and $X \subseteq [n]$, we define $A^X = [y_1,\ldots,y_n]$ to be the hyperplane with $y_i = x_i$ for all $i \in F(A) \cap X$, and $y_i = *$ otherwise. Note that $(A^X)^Y = A^{X \cap Y}$ for every $X, Y \subseteq [n]$.

The first step in the proof of Theorem~\ref{prop:moments} is the following easy bound on the $\Pr_k$-measure of a $Q_k$-measurable hyperplane.

\begin{lemma}\label{lem:AP}
Let\/ $a \le k \le n$, and let\/ $A$ be a\/ $Q_k$-measurable hyperplane. If $F(A) = I \cup J$, where $I \subseteq [a]$ and $J \subseteq [a+1,k]$, then
\begin{equation}\label{e:APdist}
  \Pr_k(A) \le c(I) \nu(J).
\end{equation}
\end{lemma}

\begin{proof}
The proof is by induction on~$k$. Note first that for $k = a$ the conclusion follows immediately from the definition~\eqref{def:c} of the function $c$, since $\nu(\emptyset) = 1$. So let $k \in [a+1,n]$, and assume that the claimed bound holds for~$\Pr_{k-1}$.

Note first that if $k \not\in F(A)$ then $A$ is $Q_{k-1}$-measurable, and so the claimed bound follows immediately by~\eqref{obs:Qmeas} and the induction hypothesis. So assume that $k \in F(A)$, and observe that, by~\eqref{e:dom}, we have
\[
 \Pr_k( A ) \le \frac{1}{1-\delta_k} \Pr_{k-1}( A ) = \frac{1}{(1-\delta_k)|S_k|} \Pr_{k-1}\big( A^{[k-1]} \big).
\]
since the probability measure $\Pr_{k-1}$ is extended uniformly on each fibre. Since $A^{[k-1]}$ is $Q_{k-1}$-measurable, by the induction hypothesis we have
\[
\Pr_{k-1}\big( A^{[k-1]} \big) \le  c(I) \nu( J \setminus \{k\}),
\]
and so, recalling the definition~\eqref{def:nu} of the function $\nu$, the claimed bound follows.
\end{proof}

We will next prove the following bound on the $t$th moments $\E_{k-1}\big[ \alpha_k(x)^t \big]$.

\begin{lemma}\label{lem:moments}
For each\/ $t \in \N$ we have
$$\E_{k-1}\big[ \alpha_k(x)^t \big] \le \frac{1}{|S_k|^t} \sum_{F_1, \dots , F_t \in \cN_k} c\big((F_1\cup \dots \cup F_t) \cap [a] \big) \cdot \nu\big((F_1\cup \dots \cup F_t) \cap [a+1, k-1] \big).$$
\end{lemma}

\begin{proof}
Observe first that, for each $x \in Q_{k-1}$, we have
$$\alpha_k(x) \, = \frac{1}{|S_k|} \sum_{y \in S_k} \id\big[ (x,y) \in B_k \big] \le \frac{1}{|S_k|} \sum_{y \in S_k} \sum_{F \in \cN_k} \id\big[ (x,y) \in A_F \big],$$
by the union bound, and the definitions~\eqref{def:B_i} and~\eqref{def:alpha} of $B_k$ and~$\alpha_k$. Note that, given $x \in Q_{k-1}$ and $F \in \cN_k$, there exists $y \in S_k$ with $(x,y) \in A_F$ if and only if $x \in A_F^{[k-1]}$, and moreover such a $y$ (if it exists) is unique. It follows that
\[
 \alpha_k(x) \le \frac{1}{|S_k|} \sum_{F \in \cN_k} \id\big[ x \in A_F^{[k-1]} \big].
\]
Note also that if $A_1$ and $A_2$ are hyperplanes, then $A_1 \cap A_2$ is either the empty set, or a hyperplane with set of fixed coordinate indices $F(A_1) \cup F(A_2)$. Therefore, by Lemma~\ref{lem:AP}, we have
\begin{align*}
\E_{k-1}\big[ \alpha_k(x)^t \big] & \le \frac{1}{|S_k|^t} \sum_{F_1,\dots, F_t \in \cN_k} \Pr_{k-1}\big( x \in A_{F_1}^{[k-1]} \cap \dots \cap A_{F_t}^{[k-1]} \big)\\
& \le \frac{1}{|S_k|^t} \sum_{F_1, \dots, F_t \in \cN_k} c\big( (F_1\cup \dots \cup F_t) \cap [a] \big) \cdot \nu\big((F_1\cup \dots \cup F_t) \cap [a+1, k-1] \big),
\end{align*}
as required.
\end{proof}

The claimed bounds on $M_k^{(1)}$ and $M_k^{(2)}$ now follow easily.

\begin{proof}[Proof of Theorem~\ref{prop:moments}]
By Lemma~\ref{lem:moments}, we have
\begin{align*}
\E_{k-1}\big[\alpha_i(x)^t\big] & \le \frac{1}{|S_k|^t} \sum_{F_1, \dots, F_t \in \cN_k} c\big( (F_1\cup \dots \cup F_t) \cap [a] \big) \cdot \nu\big((F_1\cup \dots \cup F_t) \cap [a+1, k-1] \big)\\
 &\le \frac{1}{|S_k|^t} \sum_{I \subseteq [a]} \sum_{J \subseteq [a+1,k-1]} \sum_{\substack{X_1, \dots , X_t \subseteq [k-1] \\ X_1 \cup \dots \cup X_t = I \cup J}} c(I) \nu(J)\\
 &= \frac{1}{|S_k|^t} \sum_{I \subseteq [a]} \sum_{J \subseteq [a+1,k-1]} (2^t-1)^{|I| + |J|} c(I) \nu(J)\, = \,  \frac{c_{k-1}(2^t-1)}{|S_k|^t},
\end{align*}
which proves~\eqref{eq:M:bounds}. To prove~\eqref{eq:hyperplane:removal}, suppose that $\cF(\cA)$ contains no singletons, and observe that, by Lemma~\ref{lem:moments}, we have
\begin{align*}
|S_k|^2 \E_{k-1}\big[ \alpha_k(x)^2 \big] & \le \sum_{F_1, F_2 \in \cN_k} c\big( (F_1 \cup F_2) \cap [a] \big) \cdot \nu\big((F_1 \cup F_2) \cap [a+1, k-1] \big)\\
 & \le \, \sum_{I \subseteq [a]} \sum_{J \subseteq [a+1,k-1]} \sum_{\substack{\emptyset \ne X_1, X_2 \subseteq [k-1] \\ X_1 \cup X_2 = I \cup J}} c(I) \nu(J)\\
 & = \, 1 + \sum_{I \subseteq [a]} \sum_{J \subseteq [a+1,k-1]} \big( 3^{|I| + |J|} - 2 \big) c(I) \nu(J)\\
 &= \, c_{k-1}(3)-2c_{k-1}(1)+1,
\end{align*}
as required.
\end{proof}

\section{Proof of Theorem~\ref{thm:GMM}}\label{sec:proofGMM}

In order to deduce Theorem~\ref{thm:GMM} from Theorems~\ref{thm:general} and~\ref{prop:moments},
it will suffice to show that there is an appropriate choice of $C$ and $\delta_1,\delta_2,\dots,\delta_n$ such that $\mu_n > 0$.

\begin{proof}[Proof of Theorem~\ref{thm:GMM}]
Let $(q_k)_{k \ge 1}$ be a sequence of integers with $\liminf_{k \to \infty} q_k / k > 3$, and let $N \in \N$ and $\eps > 0$ be such that $q_k > (3 + \eps)k$ for all $k \ge N$. Let $C = C(N,\eps)$ be sufficiently large, let $n \in \N$, and for each $k \in [n]$, let $S_k$ be a set of size $q_k$. We will show that if $\cA = \{A_F : F \in \cF \}$ is a finite collection of hyperplanes in $Q = S_1 \times \dots \times S_n$, no two of which are parallel, and $F(A) \not\subseteq [C]$ for every $A \in \cA$, then $\cA$ does not cover $Q$.

Fix $\delta_1 = \cdots = \delta_n = \eps/6$, and assume (without loss of generality) that $\eps$ is sufficiently small. We will start with the uniform probability measure $\Pr_0$ on $Q$, and construct inductively the probability measures $\Pr_k$ as described in Section~\ref{sec:sieve}. By Theorem~\ref{thm:general} it suffices to show that
\[
\sum_{k=1}^n \frac{M_k^{(2)}}{4\delta_k(1-\delta_k)} < 1.
\]
To prove this, note first that $M_k^{(2)}=0$ for all $1\le k \le C$, since $F(A) \not\subseteq [C]$ for every $A \in \cA$. So let $C < k \le n$, and observe that, by Theorem~\ref{prop:moments}, we have
$$M_k^{(2)} \le \, \frac{c_{k-1}(3)}{|S_k|^2} = \frac{1}{|S_k|^2} \prod_{j = 1}^{k-1} \bigg(1 + \frac{3}{(1-\delta_j)|S_j|} \bigg).$$
Now, since $|S_j| = q_j > (3 + \eps)j$ for all $j \ge N$, and by our choice of $\delta_j$, it follows that
$$\prod_{j = N}^{k-1} \bigg(1 + \frac{3}{(1-\delta_j)|S_j|} \bigg) \le \exp\bigg( \sum_{j = N}^{k-1} \frac{3}{(1-\eps/6)(3 + \eps) j} \bigg) \le k^{1-\eps/9}.$$
Moreover, $\prod_{j = 1}^{N-1} \big(1 + \frac{3}{(1-\delta_j)|S_j|} \big) \le 3^N$. Hence, assuming that $C \ge N$ (so $|S_k| \ge 3k$), we have
$$\sum_{k=1}^n \frac{M_k^{(2)}}{4\delta_k(1-\delta_k)} \le \, \sum_{k = C}^n \frac{3^N \eps^{-2}}{k^{1+\eps/9}} < \, 1$$
if $C = C(N,\eps)$ is sufficiently large, as required.
\end{proof}

We will next show that the condition on the sequence $(q_k)_{k \ge 1}$ in Theorem~\ref{thm:GMM} is close to best possible. To be precise, we will prove the following proposition.

\begin{prop}\label{prop:almost:sharp}
There exists a sequence of integers\/ $(q_k)_{k \ge 1}$ with $q_k \ge 2$ for all $k \in \N$ and
$$\liminf_{k \to \infty} \frac{q_k}{k} = 1$$
such that the following holds. For each $C > 0$, there exists $n \in \N$ and a collection\/ $\cA$ of hyperplanes that cover\/~$Q := [q_1] \times \dots \times [q_n]$, no two of which are parallel, and with $F(A) \cap [C] = \emptyset$ for every\/ $A \in \cA$.
\end{prop}

The first step is the following simple lemma.

\begin{lemma}\label{lem:greedy:covering}
Let\/ $n \ge 3$, and let $q_1,\ldots,q_n \ge 2$ be a sequence of integers such that
$$\prod_{k = 1}^n \bigg( 1 + \frac{1}{q_k} \bigg) \, \ge \, n \log n.$$
Then\/ $Q= [q_1] \times \dots\times [q_n]$ can be covered with hyperplanes, no two of which are parallel.
\end{lemma}

\begin{proof}
The proof is by induction on $n$, so first let $n = 3$, and note that if $2 \le q_1\le q_2\le q_3$ satisfy $\prod_{k = 1}^3 (1 + q_k^{-1}) > 3\log 3$, then $q_1 = q_2 = 2$. Now observe that $[2] \times [2]$ (and hence $[2]\times[2] \times[q_3]$) can be covered by hyperplanes, no two of which are parallel.

For the induction step, observe first that, by the induction  hypothesis, if $\prod_{k=1}^{n-1}(1+ q_k^{-1})\ge (n-1)\log(n-1)$ then we can find hyperplanes (with fixed coordinates in $[n-1]$) which cover~$Q$. We may therefore assume that
$$1 + \frac{1}{q_n} > \frac{n\log n}{(n-1)\log(n-1)} > 1 + \frac{1}{n},$$
and hence (without loss of generality) that $2 \le q_1 \le \dots \le q_n < n$.

We now cover $Q$ greedily: for each set $\emptyset \ne F \subseteq [n]$ in turn we choose a hyperplane $A_F$ with fixed coordinates $F$ so as to cover as much of the remaining (uncovered) subset of $Q$ as possible. Since $Q$ can be partitioned into exactly $\prod_{k \in F} q_k$ such hyperplanes, there must exist some choice of $A_F$ that covers at least a proportion $\prod_{k \in F} q_k^{-1}$ of the remaining set. Thus, after all the hyperplanes have been chosen, the remaining set has size at most
\begin{align*}
|Q| \prod_{\emptyset \ne F \subseteq [n]} \bigg( 1 - \prod_{k \in F} \frac{1}{q_k} \bigg)
& \le 
|Q| \exp\bigg( - \sum_{\emptyset \ne F \subseteq [n]} \prod_{k \in F} \frac{1}{q_k} \bigg)\\
& = \exp\bigg( 1 + \sum_{k = 1}^n \log q_k - \prod_{k = 1}^n \bigg( 1 + \frac{1}{q_k} \bigg) \bigg).
\end{align*}
Now simply observe that
$$1 + \sum_{k = 1}^n \log q_k - \prod_{k = 1}^n \bigg( 1 + \frac{1}{q_k} \bigg) < 0,$$
since $1 + \sum_{k = 1}^n \log q_k \le 1 + n\log(n-1) < n\log n$, whereas $\prod_{k=1}^n ( 1 + q_k^{-1}) \ge n\log n$, by assumption. It follows that the number of uncovered points is less than $1$, as required.
\end{proof}

We can now easily deduce Proposition~\ref{prop:almost:sharp}.

\begin{proof}[Proof of Proposition~\ref{prop:almost:sharp}]
Assume that $C$ is sufficiently large, and set
$$q_k := \bigg\lfloor \bigg( 1 - \frac{2}{\log n} \bigg) n \bigg\rfloor$$
for each $k > C$. Observe that $\lim_{k \to \infty} q_k/k = 1$, and that
\begin{align*}
\prod_{k = C+1}^n \bigg( 1 + \frac{1}{q_k} \bigg)
& = \exp\bigg( \sum_{k = C+1}^n \frac{1}{q_k} + \frac{O(1)}{q_k^2} \bigg) = \exp\bigg( \sum_{k = C+1}^n \bigg( \frac{1}{k} + \frac{2}{k\log k} \bigg) + O(1) \bigg)\\
& = \exp\Big( \log n + 2\log\log n + O_C(1) \Big) = \Omega\big( n(\log n)^2 \big).
\end{align*}
Thus, for all sufficiently large $n$, we have
\[
 \prod_{k = C+1}^n \bigg( 1 + \frac{1}{q_k} \bigg) \ge (n - C) \log(n - C),
\]
and it follows, by Lemma~\ref{lem:greedy:covering}, that we can cover $[q_{C+1}] \times \dots \times [q_n]$ with hyperplanes, no two of which are parallel. But this implies that we can cover $[q_1]\times\dots\times[q_n]$ with hyperplanes whose fixed coordinates do not intersect~$[C]$, as required.
\end{proof}

\section{The Erd\H{o}s--Selfridge problem}\label{sec:proofSFES}

In this section we will prove Theorem~\ref{thm:GSFES} (and hence also Theorem~\ref{thm:SFES}). To do so, we will again apply the sieve introduced in Section~\ref{sec:sieve}, but this time we will need to choose the various parameters much more carefully. In particular, we will deal with the primes in three groups: first the set $\{3,5,7,11\}$, then the primes between $13$ and $73$, and finally the primes larger than $73$. We will discuss these in reverse order, so as to motivate the bounds we prove.

Let $\cB$ be a collection of hyperplanes in $P := [3] \times [5] \times \dots \times [p_n]$, no two of which are parallel. Our aim is to show that $\cB$ does not cover $P$. To do so, we will in fact apply our sieve to a modified collection, obtained by removing the co-dimension~$1$ hyperplanes, as described after the statement of Theorem~\ref{prop:moments}, for all primes $p \le 73$. After doing so, we obtain a collection $\cA$ of hyperplanes in $Q = S_2 \times \dots \times S_n$, where $S_k = [p_k - 1]$ for each $2 \le k \le 21$, and $S_k = [p_k]$ for each $22 \le k \le n$, such that no two hyperplanes in $\cA$ are parallel, and if $F(A) = \{i\}$ for some $A \in \cA$ then $i \ge 22$. We remark that we will use some results from~\cite{DUS} to deal with the large primes, and our indexing of the sets $S_k$ is chosen to avoid a conflict with the notation used there. It will also be convenient (see Section~\ref{sec:smallprimes}, below) to assume (as we may) that $A \not\subseteq B$ for any $A,B \in \cA$ with $A \ne B$. 

\subsection{The primes greater than 73}\label{sec:bigprimes}

For large primes, it will suffice to apply the results of~\cite[Section~6]{DUS}. 
To state the results we will use, let us first recall some notation. Assume that we have chosen $\delta_6,\ldots,\delta_{21}$ and some probability distribution $\Pr_a = \Pr_5$ supported on 
$$R_5 \subseteq Q_5 := S_2 \times \cdots \times S_5.$$ 
Now, noting that $p_{21} = 73$, set $\kappa := c_{21}(3)$ and define
\begin{equation}\label{fk:def}
f_k = f_k(\cA) := \frac{\kappa}{\mu_k} \prod_{21 < i \le k}\bigg( 1+\frac{3p_i-1}{(1-\delta_i)(p_i-1)^2} \bigg)
\end{equation}
for each $k \ge 21$, where the constants $\{ \delta_i : i > 21\}$ will be chosen later, and recall that 
$$\mu_k = 1 - \sum_{i = 6}^k \Pr_i(B_i)$$
for each $5 \le k \le n$. In order to apply the results stated below, we need to check that condition~(20) of~\cite{DUS}, which states that
\begin{equation}\label{kappa:def:assumption}
 M^{(2)}_k \le \frac{\kappa}{(p_k - 1)^2} \prod_{21 < i < k}\bigg( 1+\frac{3p_i-1}{(1-\delta_i)(p_i-1)^2} \bigg) = \frac{\mu_{k-1} f_{k-1}}{(p_k - 1)^2},
\end{equation}
is satisfied for every $k > 21$. To see this, note that
$$c_k(3) = c_{k-1}(3) \bigg( 1 + \frac{3}{(1-\delta_k)p_k} \bigg) \le c_{k-1}(3) \bigg( 1 + \frac{3p_k - 1}{(1-\delta_k)(p_k - 1)^2} \bigg)$$
for every $k \ge 21$, and observe that therefore, by Theorem~\ref{prop:moments}, we have
$$M_k^{(2)} \le \frac{c_{k-1}(3)}{|S_k|^2} \le \frac{c_{21}(3)}{|S_k|^2} \prod_{21 < i < k}\bigg( 1+\frac{3p_i-1}{(1-\delta_i)(p_i-1)^2} \bigg),$$
as required, since $|S_k| = p_k > p_k - 1$. The following theorem, which gives an almost optimal termination criterion when $k$ is large, was proved in~\cite{DUS}.


\begin{theorem}\label{thm:f:sufficient}
Let\/ $k \ge 10$. If\/ $\mu_k > 0$ and\/ $f_k(\cA) \le (\log k + \log\log k - 3)^2k$, then the collection of hyperplanes~$\cA$ does not cover\/~$Q$.
\end{theorem}

Using Theorem~\ref{thm:f:sufficient}, one can now compute (see the discussion in~\cite[Section~6]{DUS}) weaker sufficient conditions on $f_k$ for the event that the uncovered set is non-empty. In particular, we will use the following result, cf.~\cite[Corollary~6.3]{DUS}.

\begin{corollary}\label{cor:gk}
If $f_{21}(\cA) \le 138.877$, then the collection of hyperplanes~$\cA$ does not cover\/~$Q$.
\end{corollary}


In order to prove Theorem~\ref{thm:GSFES}, it will therefore suffice to show that we can choose the probability distribution $\Pr_5$ and constants $\delta_6,\ldots,\delta_{21}$ such that $f_{21}(\cA) \le 138.877$.

\subsection{The primes between $13$ and $73$}\label{sec:mediumprimes}

We will next deduce from Corollary~\ref{cor:gk} a sufficient condition\footnote{A rough diagram of the pairs $(c_5(1),c_5(3))$ that were sufficient to prove the required bound on $f_{21}(\cA)$ was determined. Based on this, the linear combination $c_5(3)- 3c_5(1)/4$ appears to give the best `figure of merit' among simple linear combinations.}  on $\Pr_5$ for the event that the uncovered set is non-empty. 

\begin{lemma}\label{lem:middle:primes}
If\/ $\Pr_5$ satisfies\/ $c_5(3) - 3c_5(1) / 4 \le 9.019$, then there exists a choice of the constants\/ $\delta_6,\ldots,\delta_{21}$ such that\/ $f_{21}(\cA) < 138.874$.
\end{lemma}

\begin{proof}
Set $\hat\mu_5 := \mu_5 = 1$ and define
\begin{equation}\label{eq:mu:hat}
\hat\mu_k 
\, := \, \hat\mu_{k-1} - \frac{c_{k-1}(3) - 2c_{k-1}(1) + 1}{ 4 \delta_k(1 - \delta_k) |S_k|^2 }
\end{equation}
for each $k \in \{6,\ldots,21\}$. Recall that, by~\eqref{eq:bound:on:Bk} and Theorem~\ref{prop:moments}, we have
$$\Pr_k(B_k) \le \frac{M_k^{(2)}}{4\delta_k(1-\delta_k)} \le \frac{c_{k-1}(3)-2 c_{k-1}(1)+1}{4\delta_k(1-\delta_k)|S_k|^2},$$
so $\hat\mu_k \ge \mu_k$, for each $k \in \{6,\ldots,21\}$, and hence $f_{21}(\cA) = c_{21}(3)/\mu_{21} \le c_{21}(3)/\hat\mu_{21}$.

Now, observe that
\begin{equation}\label{eq:ck:recursion}
c_k(x) = c_{k-1}(x) \bigg( 1 + \frac{x}{(1-\delta_k)(p_k - 1)} \bigg)
\end{equation}
for each $x \in \{1,3\}$ and $k \in \{6,\ldots,21\}$, and therefore (for fixed $\delta_k$) we may write $c_{21}(3) = g\big( c_5(3) \big)$ and $\hat\mu_{21} = h\big( c_5(1), c_5(3) \big)$ as functions of $c_5(1)$ and $c_5(3)$. Moreover, the function $g(x)$ is increasing, and it follows from~\eqref{eq:mu:hat} and~\eqref{eq:ck:recursion} that the function $h(x,y)$ is increasing in~$x$ and decreasing in~$y$. It follows that $c_{21}(3)/\hat\mu_{21}$ is increasing in $c_5(3)$ and decreasing in $c_5(1)$, provided $\hat\mu_{21} > 0$. Thus, to bound $f_{21}(\cA)$ from above for all pairs $(c_5(1),c_5(3))$ with $c_5(3) - 3c_5(1)/4 \le 9.019$, it is enough to bound $c_{21}(3)/\hat\mu_{21}$ for a finite set of pairs $(c_5(1),c_5(3))$ that `dominate' the region $c_5(3) - 3c_5(1)/4 \le 9.019$, and satisfy $\hat\mu_{21} > 0$.

To do this, note that $c_5(1) \ge 1$ and $c_5(3)-1 \ge 3 (c_5(1)-1)$, by~\eqref{def:ck}. We may therefore assume that $c_5(1)\le 5$, since otherwise $c_5(3) - 3c_5(1)/4 \ge 9c_5(1)/4 - 2 > 9.019$. We therefore only need to cover the part of the region $c_5(3) - 3c_5(1)/4 \le 9.019$ with $1 \le c_5(1) \le 5$. We do so by looping through values of $c_5(1)$ from 1 to 5 in steps of $10^{-4}$, i.e., we check the point
$$\big( u(i),v(i) \big) := \bigg( \frac{i}{10^4}, \, 9.019 + \frac{3(i+1)}{4  \cdot 10^4} \bigg)$$
in the $(c_5(1),c_5(3))$-plane for each $10^4 \le i \le 5 \cdot 10^4$. Note that $\big( u(i),v(i) \big)$ dominates the set 
$$I(i) := \Big\{ \big( c_5(1),c_5(3) \big) : c_5(3) - 3c_5(1)/4 \le 9.019 \textup{ and } 10^{-4}i \le c_5(1) \le 10^{-4}(i+1) \Big\},$$ 
in the sense that if $(x,y) \in I(i)$ then $x \ge u(i)$ and $y \le v(i)$, so (by the monotonicity properties proved above) any upper bound on $c_{21}(3)/\hat\mu_{21}$ that holds at the point $( u(i),v(i) )$ applies to all points of $I(i)$. Hence, if an upper bound for $c_{21}(3)/\hat\mu_{21}$ holds for each pair $( u(i),v(i) )$ in the range above, then it holds whenever $\Pr_5$ satisfies $c_5(3) - 3c_5(1)/4 \le 9.019$.

Now for each $10^4 \le i \le 5 \cdot 10^4$ we choose the constants $\delta_6,\ldots,\delta_{21} \in(0,1/2]$ so as to minimize the ratio $c_{21}(3)/\hat\mu_{21}$ after processing the prime $73$. The optimization of the $\delta_k$ was made by first taking a heuristic choice, and then performing coordinate-wise optimization repetitively on each $\delta_k$ until the value of $c_{21}(3)/\hat\mu_{21}$ converged (which it did fairly rapidly). We also checked that $\hat\mu_{21} > 0$ for each of these points.

The maximum value of $c_{21}(3)/\hat\mu_{21}$ obtained for any of these points was $138.873682$, and hence $f_{21}\le 138.874$ for any $\Pr_5$ such that $c_5(3) - 3c_5(1)/4 \le 9.019$, as required.
\end{proof}

\subsection{Constructing the measure $\Pr_5$}\label{sec:smallprimes}

It remains to construct a measure on the uncovered set $R_5 \subseteq S_2 \times \cdots \times S_5$ obtained after removing all hyperplanes corresponding to arithmetic progressions whose moduli involve only the primes~3, 5, 7 and~11. 
By Corollary~\ref{cor:gk} and Lemma~\ref{lem:middle:primes}, it will suffice to prove the following lemma.

\begin{lemma}\label{lem:small:primes}
 For each\/ $\cA$ as above, there exists\/ $\Pr_5$ such that\/ $c_5(3) - 3c_5(1) / 4 \le 9.019$.
\end{lemma}

\begin{proof}
Let us write $\cF$ for the collection of subsets of $\{2,3,4,5\}$ with $|F| \ge 2$. For each set $F \in \cF$, we need to choose a hyperplane $A_F$ with fixed set $F$, and for each such family we need to construct a measure $\Pr_5$ supported on the uncovered set. Not surprisingly, there are far too may configurations to deal with easily, so we need to make a few reductions.

Recall first that (by assumption) no hyperplane in $\cA$ is contained in another. Moreover, note that we only need to study configurations `up to isomorphism', in the following sense. Let us write $F < F' $ if $\sum_{i \in F} 2^i < \sum_{i\in F'} 2^i$ (i.e., $F$ precedes $F'$ in colexicographic order), and for each $F \subseteq \{2,3,4,5\}$ and $i \in F$, define $b(F,i) \in S_i$ so that $A_F \subseteq \{ x_i = b(F,i) \}$. Now, suppose there exists a pair $(F,i)$ such that $b(F,i) \ge b(I,i) + 2$ for every $I < F$ with $i \in I$. Then we can transpose $b(F,i)$ and $b(F,i) - 1$ in $S_i$ to obtain an isomorphic configuration $\cA'$ which is lexicographically smaller, since $A'_F < A_F$, but $A'_I = A_I$ for all $I < F$.


Applying these reductions reduces the number of configurations to 6,025,640,717 which, while it represents substantial progress, is still too large to conveniently construct optimized probability distributions for each configuration. However, the main contribution to the large number of configurations comes from the choice of the `last' few hyperplanes, namely $A_{45}$, $A_{245}$, $A_{345}$, and $A_{2345}$. For example, we might have as many as $2\times 4\times 6\times 10=480$ choices for $A_{2345}$, as we are selecting a single point in~$Q_5$ (although in practice the number of choices is reduced somewhat by the comments above). Ignoring the choices for $A_{45}$, $A_{245}$, $A_{345}$, and $A_{2345}$ reduces the number of configurations to just 7637, which is far more manageable.

Our strategy is therefore as follows. We first consider a choice of the hyperplanes, $A_{23}$, $A_{24}$, $A_{25}$, $A_{34}$, $A_{35}$, $A_{234}$ and $A_{235}$, without including the last four hyperplanes $A_{45}$, $A_{245}$, $A_{345}$ and $A_{2345}$. In order to optimize the probability distribution on the uncovered region $R := Q_5 \setminus (A_{23} \cup \dots \cup A_{235})$, we construct a linear programming problem with variables $x_r$ for each $r\in R$ representing the probability of the atom~$\{r\}$. For each non-empty set $I \subseteq \{2,3,4,5\}$, and each hyperplane $H$ in $Q_5$ with $F(H) = I$, we include the constraint
$$\sum_{r \in R \cap H}  x_r \le c_I,$$
where the $c_I$ are new variables giving upper bounds on the $c(I)$, cf.~\eqref{def:c}. (Note that we include the constraints corresponding to sets not in $\cF$, since we need to bound $c(I)$ for all subsets $I \subseteq \{2,3,4,5\}$.)
We also add the constraints
\[
 x_r\ge 0 \qquad \text{and} \qquad \sum_{r\in R} x_r = 1
\]
to ensure that we have a probability measure supported on~$R$, and then minimize
\begin{equation}\label{eq:minimise:this}
 \sum_{I \subseteq \{2,3,4,5\}} \big( 3^{|I|} - 3/4 \big) c_I ,
\end{equation}
where we define $c_\emptyset=1$. We define $\Pr_5$ to be the probability distribution corresponding to this minimum, i.e., we set $\Pr_5(r) = x_r$ for each $r \in R$. Recalling~\eqref{def:ck}, note that
$$c_5(3) - \frac{3c_5(1)}{4} = \sum_{I \subseteq \{2,3,4,5\}} \big( 3^{|I|} - 3/4 \big) c(I),$$
and observe that the minimum occurs when $c_I = c(I)$.  

We are therefore done, except for the (important) fact that we have not restricted the measure to be zero
on the set
$$U : = A_{45}\cup A_{245} \cup A_{345} \cup A_{2345}.$$ 
To do so, we simply remove the measure from the (unknown) set $U$, and uniformly rescale the measure to again give a probability measure. We claim that this can increase the value of $c_5(3) - 3c_5(1)/4$ to at most
\begin{equation}\label{e:onepointrem}
 \frac{c_5(3) - 3c_5(1)/4 - p/4}{1 - p},
\end{equation}
where $p$ is the probability assigned to~$U$. To see this, observe that removing the measure on $U$ does not increase any $c(I)$, and decreases $c(\emptyset)$ by~$p$. Since $c_5(3) - 3c_5(1)/4$ is a positive linear combination of the $c(I)$, with $c(\emptyset)$ occurring with coefficient $1/4$, it follows that $c_5(3)-3c_5(1)/4$
decreases by at least $p/4$. Renormalizing the measure then increases each $c(I)$ (and hence this linear combination of the $c(I)$) by a factor of $1 - p$.

\begin{table}
\begin{tabular}{lc}
Configuration & $c_5(3) - 3c_5(1)/4$\\\hline
11**, 2*1*, *22*, 121*, 1**1, *3*2, 13*3, **34, 2*31, *232, 1233 & 9.018070\\
11**, 2*1*, *22*, 121*, 1**1, *3*2, 13*3, **34, 2*33, *232, 1233 & 9.018070\\[+2ex]
\end{tabular}
\caption{Configurations on $Q_5$ with $c_5(3) - 3c_5(1)/4 \ge 9.018$.}\label{t:configs}
\end{table}

To complete the proof, we bound the probability $p$ of the unspecified set $U$ by 
$$p \le c_{45}+c_{245}+c_{345}+c_{2345},$$
where the $c_I$ are the bounds on $c(I)$ given by the linear programming problem. We then check that the bound
in \eqref{e:onepointrem} is less than 9.018 (this bound was chosen, after some experimentation, to be just below the worst case value given in Table~\ref{t:configs}). If so, then we proceed to the next configuration. There are 90 (out of 7637) configurations of $(A_{23},\dots,A_{235})$ where this fails. For these, we loop through all choices of $A_{45}$ and perform the above calculation with just $A_{245}$, $A_{345}$, and $A_{2345}$ unspecified. From these 90 configurations we obtain 1083 configurations with $A_{45}$ included, but for only 12 of these does our bound still exceed 9.018. These 12 give rise to 312 configurations including $A_{245}$, of which 3 still exceed our bound. These 3 give rise to 216 configurations where we are forced to include $A_{345}$, but only 2 which still exceed our bound. Finally, these 2 give 142 configurations where we are forced to include all the~$A_F$, but only two have $c_5(3) - 3c_5(1)/4 \ge 9.018$, and these are listed in Table~\ref{t:configs}. We deduce that for all choices of the hyperplanes $\{ A_F : F\in \cF \}$ in $Q_5$ we can find a probability measure $\Pr_5$ on $R_5$ such that $c_5(3) - 3 c_5(1) / 4 < 9.018071$.

All calculations were performed in C using the Gurobi linear optimization package~\cite{Gurobi}
to solve the LP minimizations. With the strategy as described above, the calculations
to determine the worst case configurations took about 45 seconds on a laptop.
\end{proof}

As noted above, Theorem~\ref{thm:GSFES} follows immediately from Corollary~\ref{cor:gk} and Lemmas~\ref{lem:middle:primes} and~\ref{lem:small:primes}. In fact, since the results from~\cite{DUS} (which we used to deal with the large primes) did not require the assumption that the moduli are square-free, we actually proved something slightly stronger: if the moduli are distinct and each prime $p\le 73$ in their prime factorization occurs to a power at most~1, then at least one of the moduli must be even. In other words, the `square-free' condition is only needed on the 73-smooth part of the moduli. Of course, if we could reduce ``$p \le 73$" to ``$p < 3$" then the Erd\H{o}s--Selfridge problem would be solved, so it would be interesting to see to what extent the bound $73$ could be reduced.


\end{document}